\documentclass[11pt,reqno]{amsart}
\usepackage{amsmath, graphicx}
\usepackage{amsfonts, amssymb}
\usepackage{amsthm}
\usepackage[english]{babel}

\numberwithin{equation}{section}

\newtheorem{theorem}{Theorem}[section]
\newtheorem{corollary}[theorem]{Corollary}
\newtheorem{lemma}[theorem]{Lemma}
\newtheorem{proposition}[theorem]{Proposition}

\theoremstyle{definition}
\newtheorem{definition}[theorem]{Definition}
\newtheorem{remark}[theorem]{Remark}

\newcommand{\R}{{\mathbb R}}
\newcommand{\N}{{\mathbb N}}
\newcommand{\length}{{\rm length}}
\newcommand{\dist}{{\rm dist}}
\newcommand{\cf}{{\it cf.}}
\newcommand{\ie}{{\it i.e.}}

\long\def\forget#1\forgotten{}

\title[self-contracted planar curves and gradient orbits of convex
  functions]
{Asymptotic behaviour of self-contracted planar curves and gradient orbits of convex functions}
\author[A.~Daniilidis]{Aris Daniilidis$^{*}$}

\address{Aris Daniilidis \newline
Departament de Matem\`{a}tiques, C1/320, Universitat Aut\`{o}noma de
Barcelona, E--08193 Bellaterra (Cerdanyola del Vall\`{e}s), Spain.
\newline Universit\'e Fran\c{c}ois Rabelais, Tours. Laboratoire de
Math\'ematiques et  Physique Th\'eorique. CNRS, UMR 6083.
F\'ed\'eration de Recherche Denis Poisson (FR 2964). Parc de
Grandmont, 37400 Tours, France}

\email{arisd@mat.uab.es}
\urladdr{http://mat.uab.es/\symbol{126}arisd}

\author[O.~Ley]{Olivier Ley}
\author[S.~Sabourau]{St\'ephane Sabourau}

\address{
Olivier Ley and St\'ephane Sabourau \newline Universit\'e
Fran\c{c}ois Rabelais, Tours. Laboratoire de Math\'ematiques et
Physique Th\'eorique. CNRS, UMR 6083. F\'ed\'eration de Recherche
Denis Poisson (FR 2964). Parc de Grandmont, 37400 Tours, France}
\email{ley@lmpt.univ-tours.fr}
\urladdr{http://www.lmpt.univ-tours.fr/\symbol{126}ley}

\email{sabourau@lmpt.univ-tours.fr}
\urladdr{http://www.lmpt.univ-tours.fr/\symbol{126}sabourau}

\date{\today}



\thanks{$^{*}$ Supported by the MEC Grant No. MTM2005-08572-C03-03 (Spain).}

\begin{document}


\forget Rien n'apparait dans le document ! \forgotten


\begin{abstract}
We hereby introduce and study the notion of self-contracted curves,
which encompasses orbits of gradient systems of convex and
quasiconvex functions. Our main result shows that bounded
self-contracted planar curves have a finite length. We also give an
example of a convex function defined in the plane whose gradient
orbits spiral infinitely many times around the unique minimum of the
function.
\end{abstract}

\maketitle

\tableofcontents

\section{Introduction}

This work is mainly devoted to the study of the length of bounded
trajectories of the gradient flow of convex (or quasiconvex)
functions in the plane. The motivation for this study comes from a
well-known result due to S.~\L ojasiewicz (see \cite{Loja93}),
asserting that if $f:\mathbb{R}^{n}\rightarrow \mathbb{R}$ is a
real-analytic function and $\bar{x}\in f^{-1}(0)$ is a critical
point of $f$, then there exist two constants $\rho\in [1/2,1)$ and $C>0$
such that
\begin{eqnarray}\label{loja-ineq}
||\nabla f(x)||\geq C \, |f(x)|^\rho
\end{eqnarray}
for all $x$ belonging to a neighborhood $U$ of $\bar{x}$. An
immediate by-product is the finite length of the orbits of the
gradient flow of $f$ lying in $U$. The proof is straightforward
using \eqref{loja-ineq} and is illustrated below: let
$\gamma:[0,+\infty)\to U$ be a gradient trajectory of~$f$, that is,
$\dot\gamma(t)=-\nabla f(\gamma(t))$. Then,
\begin{eqnarray*}
-\left(\frac{1}{1-\rho}\right)\, \frac{d}{dt} \left[ f(\gamma(t))^{1-\rho} \right]
=-\langle \dot\gamma(t),\nabla f(\gamma(t)\rangle \,
f(\gamma(t))^{-\rho} = ||\dot\gamma(t)||^2 \,f(\gamma(t))^{-\rho}
\geq C \, ||\dot\gamma(t)||,
\end{eqnarray*}
yielding
\begin{eqnarray}\label{long-finie}
\length(\gamma)=\int_0^{+\infty}||\dot\gamma(t)||dt <+\infty.
\end{eqnarray}
The aforementioned inequality \eqref{loja-ineq} has been extended by
K.~Kurdyka in~\cite{Kur98} for $C^1$ functions belonging to an
arbitrary o-minimal structure (we refer to \cite{Dries-Miller96} for
the relevant definition), in a way that allows us again to deduce the
finiteness of the lengths of the gradient orbits in this more
general context. In \cite{BDL04} and~\cite{BDLS-2008}, a further
extension has been realized to encompass (nonsmooth) functions and
orbits of the corresponding subgradient systems.
\smallskip

It should be noted that in the above cases the functions enjoy an
important structural property (o-minimality) and that, for general
functions, there is no hope to prove a result like
\eqref{long-finie}. A classical example of J.~Palis and W.~De Melo
(\cite[page~14]{PalDem82}) asserts that the bounded trajectories of
the gradient flow of an arbitrary $C^{\infty}$ function need not
converge (in particular, they are of infinite length). In the
aforementioned example the critical set of the function is not
reduced to a singleton: in Section~\ref{subsec:1}, we provide another
example of a smooth function having a unique critical point towards
which all corresponding orbits converge, but again are of infinite
length.
\smallskip

The case when $f$ is a convex coercive function is particularly
interesting in view of its potential impact in numerical
optimization (see \cite{AMA2005}, \cite{BDL04}, \cite{BDLM-2008},
for example). But convex functions are far from being analytic and
they do not satisfy neither the \L ojasiewicz inequality nor its
generalized form established by Kurdyka, unless a growth condition
is assumed (see~\cite[Sections~4.2--4.3]{BDLM-2008} for a sufficient
condition and a counter-example). Nevertheless, their rigid
structure makes it natural to think that the orbits of their gradient
flow are of finite length. It is rather surprising that the answer
of this question is not yet known in the literature except in some
particular cases.
\smallskip

Let us mention that in the framework of Hilbert spaces, this has
been stated as an open problem by H.~Br\'ezis~\cite[Open problems,
p.~167]{Brezis}. In infinite dimension, R.~Bruck~\cite{bruck75}
proved that the (sub)gradient orbits of convex coercive functions
are converging towards a global minimizer of $f$ but this
convergence holds only with respect to the weak topology. Indeed,
B.~Baillon~\cite{baillon78} constructed a counterexample of a lower
semicontinuous convex function $f$ in a Hilbert space whose gradient
orbits do not converge for the norm topology. A straightforward
consequence is that these orbits have infinite length. Concurrently,
there are some cases where a convex coercive function $f:H\to \R$ is
known to have (sub)gradient orbits of finite length. This is true
when the set of minimizers of $f$ has nonempty interior in the
Hilbert space $H$ (see H.~Br\'ezis~\cite{Brezis}), or whenever $f$
satisfies a growth condition. For a detailed discussion and the
proofs of these facts, we refer to \cite[Section~3-4]{BDLM-2008}.

\smallskip

The aforementioned results do not cover the simplest case of a
convex smooth function defined in the plane and having a unique
minimum. One of the main results of this work is to prove the
following:

\begin{theorem}[Convex gradient system]\label{thm-intro-convex}
Let $f:\R^{2} \to \R$ be a smooth convex function with a unique
minimum. Then, the trajectories~$\gamma$ of the gradient system
\begin{equation*}
\dot{\gamma}(t)=-\nabla f(\gamma(t))
\end{equation*}
have a finite length.
\end{theorem}

\noindent The proof of this result does not
use the whole convexity of $f$ but, instead, rather relies on the
convexity of its level-sets. More precisely, the conclusion of Theorem~\ref{thm-intro-convex} remains
also true for the orbits of the gradient flow of a quasiconvex
function (see Corollary~\ref{Corollary_quasiconvex}).\smallskip

Actually, both results will follow as consequences of a much more
general result (Theorem~\ref{Theorem_main}) about
bounded \emph{self-contracted} planar curves, which allows us to provide
a unified framework for this study.

\begin{definition}[Self-contracted curve] \label{Definition_Self-contracted}
A curve $\gamma:I \rightarrow\mathbb{R}^{n}$ defined on an
interval~$I$ of $[0,+\infty)$ is called {\em self-contracted}, if
for every $t_{1} \leq t_{2} \leq t_{3}$, with $t_{i} \in I$, we have
\begin{equation} \label{eq:self}
\dist(\gamma(t_{1}),\gamma(t_{3})) \geq
\dist(\gamma(t_{2}),\gamma(t_{3})).
\end{equation}
In other words, for every $[a,b]\subset I,$ the map $t\in [a,b]
\mapsto \dist(\gamma(t),\gamma(b))$ is nonincreasing.
\end{definition}

We prove the following.

\begin{theorem}[Main result] \label{Theorem_main}
Every bounded continuous self-contracted planar curve~$\gamma$ is of
finite length. More precisely,
\begin{equation*}
\length(\gamma) \leq (8 \pi + 2) \, D(\gamma)
\end{equation*}
where $D(\gamma)$ is the distance between the endpoints of~$\gamma$.
\end{theorem}


Let us finally mention that, even if gradient orbits of convex
functions have finite length in the plane, they do not enjoy all the properties of the gradient orbits of real-analytic
functions. Indeed, on the one hand, the so-called Thom conjecture for the gradient orbits of real-analytic
functions holds true:
if $x_{\infty}$ denotes the limit of the orbit $\gamma(t)$, then the
secants $(\gamma(t)-x_{\infty})/||\gamma (t)-x_{\infty}||$ converge
towards a fixed direction of the unit sphere (see K.~Kurdyka,
T.~Mostowski and A.~Parusinski \cite{KMP2000}). On the other hand, as we show in
Section~\ref{subsec:2}, an analogous result fails in the convex case.
Indeed, the orbits of a convex gradient flow may turn around their
limit infinitely many times.
\smallskip

Our techniques only work in the two-dimensional case. We do not know
whether Theorem~\ref{thm-intro-convex} and Theorem~\ref{Theorem_main}
hold in greater dimension.
\smallskip

The article is organized as follows. In Section~\ref{sec:2}, we
present basic properties of self-contracted curves. In
Section~\ref{sec:v}, we decompose each polygonal approximation of a
bounded self-contracted curve in an annulus centered at its endpoint
into horizontal and vertical segments. We establish upper bounds on
the total length of the vertical segments in Section~\ref{sec:v} and
on the total length of the horizontal segments in
Section~\ref{sec:h}. The proof of the main result is presented in
Section~\ref{sec:proofs}. In Section~\ref{sec:6}, we show that the
orbits of various dynamical systems are self-contracted curves. Two
(counter)-examples are presented in Section~\ref{sec:ex}.

\bigskip

\textit{Notations}. Throughout the manuscript, we shall deal with
the finite-dimensional Euclidean space $\mathbb{R}^{n}$ equipped
with the canonical scalar product $\langle\cdot,\cdot\rangle$. We
denote by $\Vert\cdot\Vert$ (respectively, $\dist(\cdot,\cdot)$) the
corresponding norm (respectively, distance). Therefore, the distance
between two points $x$ and~$y$ of~$\R^2$ will be denoted by $\Vert
x-y\Vert,$ $\dist(x,y)$ or sometimes $|xy|.$ We also denote by
$\mathrm{dist\,}(x,S)$ the distance of a given point
$x\in\mathbb{R}^{n}$ to a set $S\subset \mathbb{R}^{n}$, by $B(x,r)$
the closed ball with center $x\in\mathbb{R}^{n}$ and radius $r>0$
and by $S(x,r)$ its boundary, that is, the sphere of the same center
and the same radius. For $0<r<R$, we denote by

\begin{equation} \label{annulus}
U(r,R):=\{ x\in\R^{n} \mid r < \| x \| \le R\}
\end{equation}
the annulus centered at the origin $O$ with outer radius $R$ and
inner radius $r$ and by $\Delta R = R-r$ its width. Let
$$[p,q]\,:=\,\{p+t(q-p) \mid t\in\lbrack0,1]\}$$
be the closed segment with endpoints $p,q\in\R^n$. A subset $S$ of
$\R^n$ is called convex, if $[p,q]\subset S$ for every $p,q\in S$.

\section{Self-contracted curves}\label{sec:2}


Throughout this paper, we shall deal with curves
$\gamma:\,I\,\rightarrow\mathbb{R}^{n}$, defined on an interval~$I$
of~$\R.$ We recall that the length of a curve $\gamma: I\to \R^n$ is
defined as
$$ \length(\gamma) = \,\sup\, \left\{
\sum_{i=1}^{k}\, \dist(\gamma(t_i),\gamma(t_{i+1})) \right\} $$
where the supremum is taken over all the finite
subdivisions~$\{t_{i}\}_{i=1}^{k+1}$ of~$I$.

We shall need the following definition.

\begin{definition}[Convergence of a curve]
A curve $\gamma:I \to \R^n$ is said to converge to a
point~\mbox{$x_0\in \R^n$} if $\gamma(t)$ converges to~$x_0$ when
$t$
goes to $t_+:=\sup I$. \\
A curve $\gamma:I \rightarrow\mathbb{R}^{n}$ is said to be
\emph{bounded}, if its image $\gamma(I)$ is a bounded subset of
$\mathbb{R}^{n}$.
\end{definition}

We start with an elementary property of self-contracted curves.

\begin{proposition}[Existence of left/right limits] \label{prop:cont}
Let $\gamma:I \mapsto\R^{n}$ be a bounded self-contracted curve and
$(a,b) \subset I$. Then, $\gamma$ has a limit in $\R^n$ when $t\in
(a,b)$ tends to an endpoint of~$(a,b)$. \\ In particular, every
self-contracted curve can be extended by continuity to the endpoints
of $I$ $($possibly equal to~$\pm\infty)$.
\end{proposition}

\begin{proof}
Since $\gamma$ lies in a compact set of~$\R^{n}$, there exists an
increasing sequence $\{t_{i}\}$ in $(a,b)$ with $t_{i} \to b$ such
that $\gamma(t_{i})$ converges to some point of~$\R^n$,
noted~$\gamma(b)^{+}$. Fix any $i,j \in \N^{*}$ and let $t_i<t<t_{i+j}.$
By~\eqref{eq:self}, we have
$$ \dist(\gamma(t),\gamma(t_{i+j}))
\leq \dist(\gamma(t_{i}),\gamma(t_{i+j})).$$ Letting $j$ go to infinity, we derive
$$ \dist(\gamma(t),\gamma(b)^{+})
\leq \dist(\gamma(t_{i}),\gamma(b)^{+})$$ which gives $\gamma(t)\to
\gamma(b)^+.$

\smallskip

Further, using the triangle inequality and
the inequality~\eqref{eq:self}, we have $$
\dist(\gamma(t_{1}),\gamma(t_{2})) \leq
\dist(\gamma(t_{1}),\gamma(t_{3})) +
\dist(\gamma(t_{3}),\gamma(t_{2})) \leq 2 \,
\dist(\gamma(t_{1}),\gamma(t_{3})). $$ Using this inequality, we can
show, as previously, that $\gamma(t)$ converges as $(a,b)\ni t \to
a$.

\smallskip

The last part of the assertion is straightforward.
\end{proof}

The following result is a straightforward consequence of
Proposition~\ref{prop:cont}.

\begin{corollary}[Convergence of bounded self-contracted curves]\label{corollary-convergence}
Every bounded self-contracted curve $\gamma:(0,+\infty)\to\R^n$
converges to some point $x_0 \in \R^2$ as $t\to+\infty$. Moreover,
the function $t\mapsto\dist(x_0,\gamma(t))$ is nonincreasing.
\end{corollary}

In the sequel, we shall assume that every self-contracted curve
$\gamma:I\mapsto\R^n$ is (defined and) continuous at the endpoints
of $I$.

\begin{remark}[Basic properties]\label{remark-basic} \ \\
(i) Inequality~\eqref{eq:self} shows that the image of a segment
$(a,b)$ by a self-contracted curve $\gamma$ lies in a ball of
radius~$\rho:=\dist(\gamma(a),\gamma(b))$. \smallskip

\noindent (ii) A self-contracted curve might not be (left/right)
continuous. A simple example is provided by the following planar
self-contracted curve: $$ \gamma(t)=\left\{
\begin{array}
[c]{ll} (t,1) & \quad\text{if }\,t\in(-\infty,0)\\ (0,0) &
\quad\text{if }t=0\\ (t,-1) & \quad\text{if }\,t\in(0,+\infty)
\end{array}
\right. $$

\noindent (iii) If $t\in (a,b)\mapsto \gamma (t)$ is a
self-contracted curve, then the curve $t\in (a,b)\mapsto \gamma
(a+b-t)$ is not necessarily self-contracted.\smallskip

\noindent (iv) Corollary~\ref{corollary-convergence} reveals that
the trajectories of a general gradient system
$$\dot{\gamma}(t)=-\nabla f(\gamma(t)), \qquad
\gamma(0)=x_{0}\in\R^n$$ might fail to be self-contracted curves.
Indeed in \cite[page 14]{PalDem82} an example of a $C^{\infty}$
function $f:\R^2\to\R$ is given, for which all trajectories of its
gradient system are bounded but fail to converge. \smallskip

\noindent (v) In Section~\ref{sec:6}, we show that whenever $f$ is
(quasi)convex, the gradient trajectories are self-contracted curves.
Thus, bounded self-contracted curves might fail to converge for the
strong topology in a Hilbert space (see Baillon's example in
\cite{baillon78}).
\end{remark}

From now on, we restrict ourselves to the two-dimensional case, and
study the asymptotic behaviour of self-contracted planar curves.

\section{Horizontal and Vertical directions} \label{sec:v}

In this section, we introduce a binary-type division of planar
segments into horizontal and vertical ones. We shall apply this
decomposition for segments issued from polygonal line approximations
of a bounded self-contracted curve. In this section, we derive an
upper bound on the total length of the vertical segments, while in
the next section we shall do the same for the total length of the
horizontal ones. Combining both results we shall thus obtain an
upper bound estimation on the total length of a bounded
self-contracted curve, establishing
Theorem~\ref{Theorem_main}.\smallskip

Fix $0<r<R$ and let $U(r,R)$ be the annulus defined in
\eqref{annulus}. Let $\sigma$ be a segment of~$U(r,R)$, not reduced
to a point. Denote by $p$ and~$q$ the endpoints of~$\sigma$ and
by~$m$ its midpoint. Switching $p$ and~$q$ is necessary, we can
assume that $q$ is closer to the origin~$O$ than~$p$, that is
$\dist(O,q)\le\dist(O,p)$. Let $\widehat{Omq}:= \theta$ be the angle
between the vectors $\overrightarrow{mO}$ and $\overrightarrow{mq}$,
\cf~Fig.~\ref{fig:1}. Note that $\theta \in
[-\frac{\pi}{2},\frac{\pi}{2}]$ (by convention, inverse-clockwise
angles are positive). \\

\setlength\unitlength{1pt} \noindent \begin{figure}[htbp] \noindent
\begin{picture}(370,170)(45,0)
\put(80,0){$\includegraphics[height=150pt]{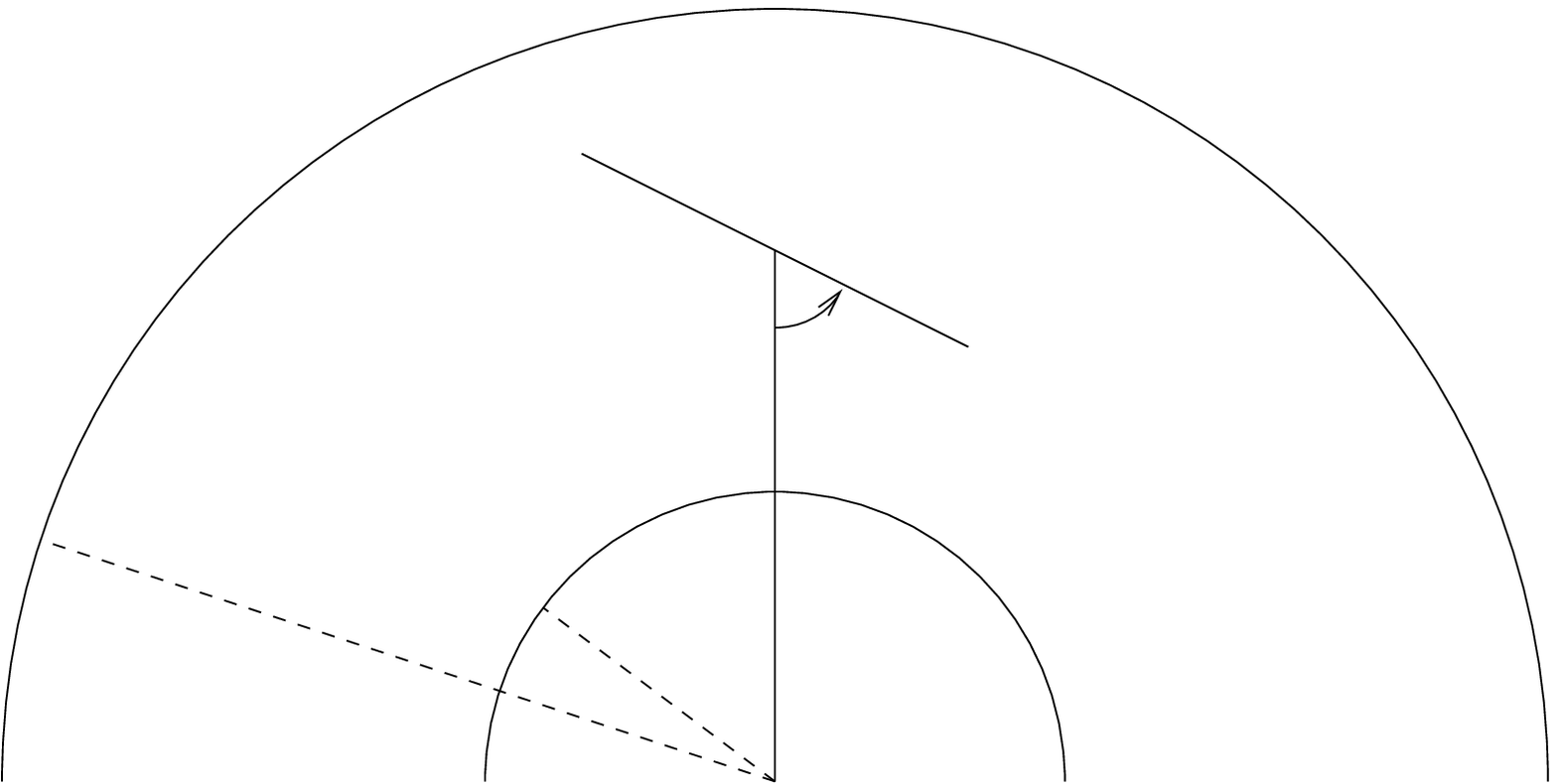}$}
\put(235,-3){$O$} \put(205,25){$r$} \put(120,20){$R$}
\put(270,90){$q$} \put(238,80){$\theta$} \put(230,110){$m$}
\put(190,130){$p$}
\end{picture}
\caption{} \label{fig:1}
\end{figure}

\begin{lemma}[Segment length estimate] \label{lem:theta}
Let $\sigma$ be a segment of~$U(r,R)$ with endpoints $p$ and~$q$
such that $\theta \neq \pm \frac{\pi}{2}$. Then $$ \length(\sigma)
\leq\, \frac{2}{\cos \theta}\, |\dist(O,p)-\dist(O,q)|. $$
\end{lemma}

\begin{proof}
Let $\bar{p}$ be the orthogonal projection of~$p$ to the line~$Om$.
Using elementary trigonometry in the right-angled
triangle~$p\bar{p}m$, we derive $$ \dist(m,\bar{p}) = \frac{1}{2}
\cos \theta \cdot \length(\sigma). $$ Hence, $$ \dist(O,p) -
\dist(O,m) \geq \frac{1}{2} \cos \theta \cdot \length(\sigma).$$
Since $\dist(O,q) \leq \dist(O,m)$, the conclusion follows.
 \end{proof}

Fix $\alpha \in (0,\frac{\pi}{2})$, let $\lambda\in (0,1)$ be such
that $\sin \alpha < \lambda < 1$, and set $r:=\lambda R$. Denote by
$$A:=U(\lambda R,R)$$ the corresponding annulus of \eqref{annulus},
with width equal to $\Delta R = (1-\lambda) R$. We now introduce a
crucial definition in the proof of our main result.

\begin{definition}[Classification of the
segments]\label{Def-classification} Let $\alpha \in
(0,\frac{\pi}{2})$, $\lambda\in (0,1)$ and $A$ as above. A
nontrivial segment~$\sigma$ of~$A$ is said to be \smallskip
\begin{itemize}
\item {\em vertical}, if $\theta$ lies in~$(-\frac{\pi}{2}+\alpha,\frac{\pi}{2}-\alpha)$
; \smallskip
\item {\em horizontal, pointing in the positive direction,} if $\theta$ lies in $[-\frac{\pi}{2},-\frac{\pi}{2}+\alpha]$
;\smallskip
\item {\em horizontal, pointing in the negative direction,} if $\theta$ lies in $[\frac{\pi}{2}-\alpha,\frac{\pi}{2}]$.
\end{itemize}
\end{definition}

For instance, the segment $[p,q]$ in~Fig.~\ref{fig:1} points in the
negative direction.

\begin{definition}[Polygonal approximation] \label{def:approx}
Let $\gamma:I\to \R^2$ be a continuous self-contracted planar curve
converging to the origin~$O$. We consider
polygonal approximations $\{\sigma_{i}\}_{i=1}^{k+1}$ of~$\gamma$ in the annulus
$A$, as follows: Let  $t_{1} < t_{2} < \cdots < t_{k+1}$ be a
sequence of points of~$I$ with $\gamma(t_{i}) \neq \gamma(t_{i+1})$
such that the restriction of~$\gamma$ to~$[t_{1},t_{k+1}]$ lies
in~$A$. Refining the subdivision if necessary, we can further assume
that for every $i\in\{1,\dots,k\}$ the segment~$\sigma_{i}$ with
endpoints~$p_{i}=\gamma(t_{i})$ and~$q_{i}=p_{i+1}=\gamma(t_{i+1})$
lies in~$A$ and that the length of~$\sigma_{i}$ is within any
desired precision~$\eta>0$ of the length of~$\gamma_{|[t_{i},t_{i+1}]}$ (this precision $\eta>0$ will be defined
at the beginning of Section~\ref{sec:h} and will only depend on
$\alpha$, $\lambda$ and~$R$). Since the function $t \mapsto
\dist(O,\gamma(t))$ is nonincreasing
(\cf~Corollary~\ref{corollary-convergence}), we can further assume
that, if $\sigma_i:=[p_i,q_i]$ is vertical, then $q_i$ is the
closest point of $\sigma_i$ to the origin. We denote
by~$m_{i}:=\frac{1}{2}(p_i+q_i)$ the midpoint of~$\sigma_{i}$.
\end{definition}

\begin{remark} \label{rmq-non-sc}
It is worth noticing that the polygonal approximation of a
self-contracted curve introduced above is no more a self-contracted
curve in general. Nevertheless, one still has
$$
\dist(p_i,p_l) \geq \dist(p_j,p_l) \quad {\rm when} \ 1\leq i\leq
j\leq l\leq k+1,
$$
which is the property we will use.
\end{remark}

The total length of the vertical segments satisfies the following
upper bound.

\begin{lemma}[Total vertical length upper bound] \label{lem:V}
If $\{\sigma_{i}\}_{i=1}^{k+1}$ is a polygonal approximation of
$\gamma$ in the annulus $A$, \cf~Definition~\ref{def:approx}, then
$$ \sum_{i \in \mathcal{V}} \length(\sigma_{i}) \leq
\frac{2}{\cos \alpha} \, \Delta R $$ where the sum is taken over all
indices $i\in \mathcal{V}\subset\{1,\dots,k+1\}$ corresponding to
the vertical segments.
\end{lemma}

\begin{proof} Let $\theta_i$ denote the angle between
$\overrightarrow{m_iO}$ and $\overrightarrow{m_iq_i}$. Since
$i\in\mathcal{V}$, it follows that $|\theta_i|<\alpha$, whence
$(\cos \theta_i)^{-1}<(\cos \alpha)^{-1}$. From
Lemma~\ref{lem:theta} (segment length estimation), we obtain
$$\sum_{i \in \mathcal{V}} \length(\sigma_{i})< \frac{2}{\cos
\alpha} \, \sum_{i \in \mathcal{V}}  \dist(O,p_{i}) -
\dist(O,q_{i}).$$

Since $\dist(O,p_{i})>\dist(O,q_i)$, for all $i\in\{1,\dots,k\}$ we
deduce $$ \sum_{i \in \mathcal{V}} \dist(O,p_{i}) - \dist(O,q_{i})
\leq \sum_{i=1}^{k} \dist(O,p_{i}) - \dist(O,q_{i}).$$ Now, since
$q_{i}=p_{i+1}$, the right-hand term is equal to
$\dist(O,p_{1})-\dist(O,q_{k})$, which is less or equal to the width
$\Delta R$ of~$A$. The proof is complete.
\end{proof}

\section{Length estimate for horizontal directions} \label{sec:h}

In this section, we keep the notations and the definitions from the
previous section. In particular, \begin{equation}\label{lambda}
\alpha\in (0,\frac{\pi}{2}),\qquad\sin \alpha<\lambda<1, \qquad
A:=U(\lambda R,R)
\end{equation}
and $\{\sigma_{i}\}_{i=1}^{k+1}$ is a polygonal approximation of
$\gamma$ in the annulus $A$, \cf~Definition~\ref{def:approx}.

We establish an upper bound on the total length of the horizontal
segments issued from the polygonal approximation of
$\gamma$. \smallskip

Let $x \in A$. The distance from the origin to the half-line~$L_{x}$
passing through~$x$ and making an angle~$\alpha>0$
with~$\overrightarrow{xO}$ is equal to $\sin \alpha \cdot
\dist(O,x)$. Thus, the half-line~$L_{x}$ intersects the
circle~$S(0,\lambda R)$ of radius~$\lambda R$ centered at the origin
at two points. These two points are noted $\pi(x)$ and~$\pi'(x)$,
with $\pi(x)$ closer to~$x$ than~$\pi'(x)$. Furthermore, the
half-line~$L_{x}$ intersects~$A$ along two segments~$\Delta_{x}$
and~$\Delta'_{x}$, where the endpoints of~$\Delta_{x}$ agree
with~$x$ and~$\pi(x)$, and one of the endpoints of~$\Delta'_x$
agrees with~$\pi'(x)$. Note that $\mathrm{min}_{x\in
A}\parallel\pi(x)-\pi'(x)\parallel>\delta>0$. The half-line~$L_{x}$
extends to a line which bounds a (closed) half-plane~$H_{x}$
containing the origin of the plane. Denote by~$\overline{H_{x}^{c}}$
the (closed) half-plane with the same boundary as $H_{x}$, not
containing the origin (see Fig.~\ref{fig:2} for an illustration of
these notations). The mappings $x\mapsto\pi(x)$ and
$x\mapsto\pi'(x)$ from~$A$ to~$S(0,\lambda R)$ are clearly
continuous, thus also uniformly continuous. Therefore, there
exists~$\eta > 0$ such that for every pair of points $x,y \in A$
which are $\eta$-close from each other (\ie \, $\dist(x,y)<\eta$),
we have
\begin{eqnarray*}
\dist(x,\pi(y))<\dist(x,\pi'(y)), \quad
\dist(\pi(x),\pi(y))<\dist(\pi(x),\pi'(x)),
\end{eqnarray*}
\begin{eqnarray}\label{antipodal}
\text{and} \quad \dist(\pi(x),\pi(y))<\dist(\pi(x),\pi'(y)).
\end{eqnarray}

\setlength\unitlength{1pt} \noindent \begin{figure}[htbp] \noindent
\begin{picture}(370,230)(45,0)
\put(130,0){$\includegraphics[height=210pt]{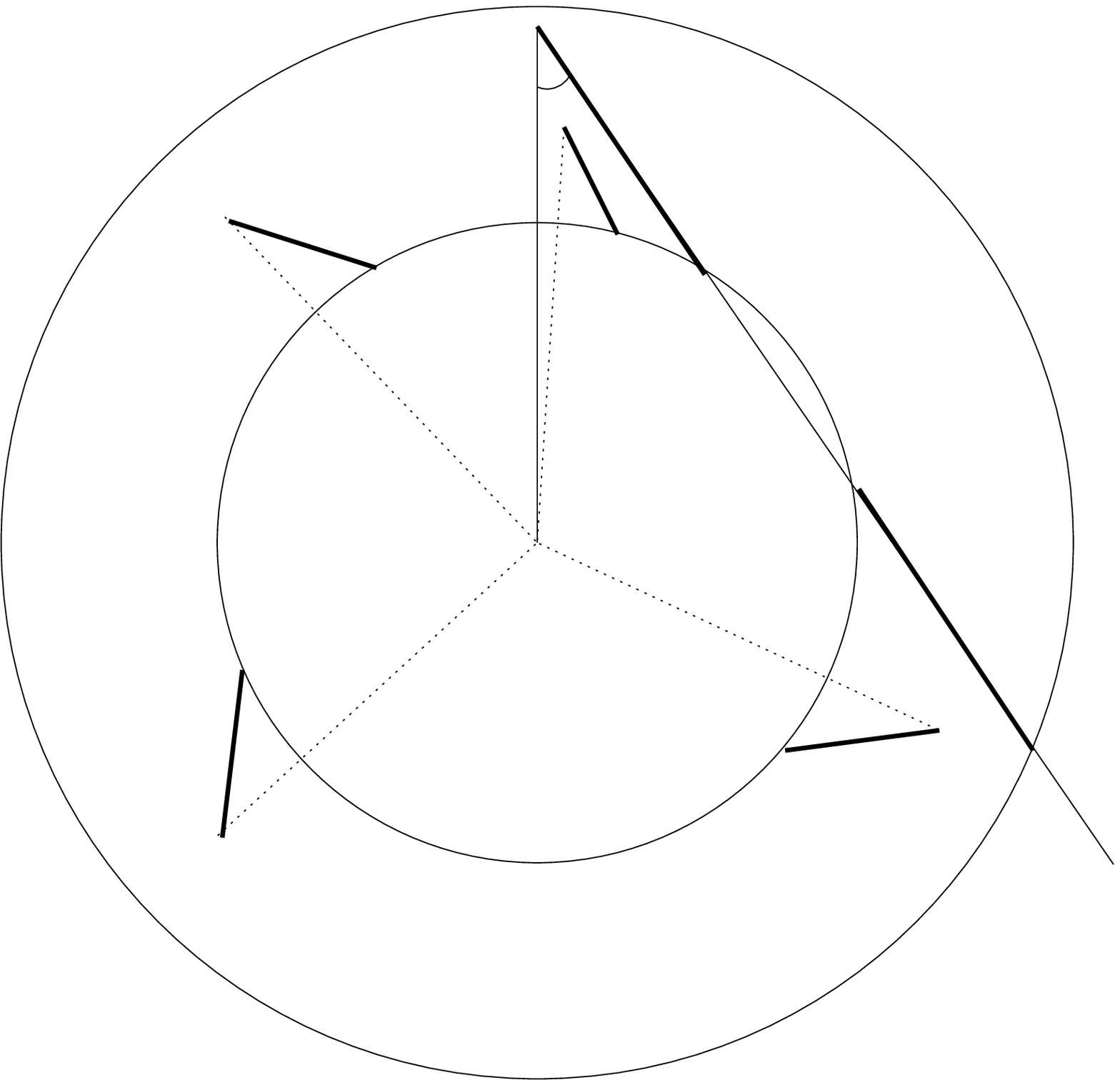}$}
\put(240,105){$O$} \put(225,203){$x$} \put(258,182){$\Delta_x$}
\put(273,160){$\pi(x)$} \put(300,118){$\pi'(x)$}
\put(319,92){$\Delta'_x$} \put(345,55){$H_x^c$} \put(348,35){$L_x$}
\put(325,35){$H_x$} \put(165,169){$y$} \put(189,169){$\Delta_y$}
\put(120,27){$S(0,R)$} \put(245,33){$S(0,\lambda R)$}
\end{picture}
\caption{} \label{fig:2}
\end{figure}

We shall further need the following technical lemmas.

\begin{lemma}[Essential disjointness of $\Delta_x$, $\Delta_y$] \label{lem:nointersec}
Let $x$ and $y$ be two distinct points of~$A$ such that
$\dist(O,x)>\dist(O,y)$. If $y \in H_{x}$, then the
segment~$\Delta_{y}$ lies in~$H_{x}$. Furthermore, $\Delta_{y}$ does
not intersect~$\Delta_{x}$, except possibly at~$y$.
\end{lemma}

\begin{proof}
Suppose first that $y$ lies in~$\Delta_{x}$. One easily checks that
the angle~$\widehat{Oz\pi(x)}$ increases  when the point~$z$ moves
from~$x$ to~$\pi(x)$ along~$\Delta_{x}$. In particular, the
angle~$\widehat{Oy\pi(x)}$ is greater than~$\alpha$. Therefore, the
segment~$\Delta_{y}$ lies in~$H_{x}$ and meets~$\Delta_x$ only
at~$y$. \smallskip

Suppose now that $y\notin \Delta_x$ and (towards a contradiction)
that $\Delta_{y}$ intersects~$\Delta_{x}$ at~$z \neq y$. Let $y'$
denote the intersection point of~$\Delta_{x}$ with the circle of
radius~$|Oy|$ centered at the origin. Then, the image of~$\Delta_{y}$
by the rotation around the origin taking~$y$ to~$y'$  does not lie
in~$H_{x}$ (the image of $z$ should lie in $H_x^c$). On the other
hand, this image agrees with~$\Delta_{y'}$ since the rotation sends
the ray $Oy$ to~$Oy'$. From the previous discussion, we conclude
that the image of~$\Delta_{y}$ is contained in~$H_{x}$. Hence a
contradiction.\smallskip

Finally, suppose that $y\notin \Delta_x$,
$\Delta_y\cap\Delta_x=\emptyset$ and $\Delta_{y}$
intersects~$\Delta_{x}'$ at~$z$. If $z=\pi'(x),$ then obviously
$\Delta_{y}$ lies in $H_x.$ Suppose now that $z\not=\pi'(x)$ (thus
$z\not= \pi(y)$). Since the angle~$\widehat{Oz\pi(y)}$ is positive
while the angle~$\widehat{Oz\pi'(x)}$ is negative, we obtain
that~$\widehat{\pi'(x)z\pi(y)}$ is positive which is not possible.
The proof is complete.
\end{proof}

\begin{lemma} [Injectivity of $\pi$] \label{lem:hc}
Let $\sigma:=[p,q]$ be an horizontal segment of~$A$, with midpoint
$m$, pointing in the positive direction. Assume $\dist(O,q)\le
\dist(O,p)$. Then,\smallskip
\begin{enumerate}
\item the restriction of~$\pi$ to~$\sigma$ is injective;\smallskip
\item if the length of~$\sigma$ is  at most~$\eta$, then the circular arc~$\pi([p,m])$ lies in~$\overline{H_{m}^{c}}$.
\end{enumerate}
\end{lemma}

\begin{proof}
Let $x,y \in [p,q]$ with $\dist(p,x)<\dist(p,y)$. Since the
horizontal segment~$\sigma$ points in the positive direction, the
angle $\widehat{xOy}$ is positive and $y$ lies in~$H_{x}$. From
Lemma~\ref{lem:nointersec} (Essential disjointness of $\Delta_x$,
$\Delta_y$), $\pi(x)$ and~$\pi(y)$ are distinct (the case $y =
\pi(x)$ is impossible since it would yield that the angle
$\widehat{\pi(x)Ox}=\widehat{yOx}$ is positive, a contradiction).
Hence the first part of the lemma follows. \smallskip

Let $x\in (m,p].$ From above, the midpoint~$m$ of~$\sigma$ lies
in~$H_x$ and the segments $\Delta_x$ and~$\Delta_m$ do not intersect
each other from  Lemma~\ref{lem:nointersec}. By definition
of~$\eta$, in view of \eqref{antipodal} the points $x$ and~$\pi(x)$
are closer to~$\pi(m)$ than to~$\pi'(m)$. Thus, the
segment~$\Delta_{x}$, which does not intersect~$\Delta_{m}$, does
not intersect~$\Delta'_{m}$ either. That is, $\Delta_{x}$ lies
either in~$H_{m}$ or in~$\overline{H_{m}^{c}}$. Since the horizontal
segment~$\sigma$ is pointing towards the positive direction, the
point~$x$ belongs to~$H_{m}^{c}$. Therefore, the same holds true for
the other endpoint~$\pi(x)$ of~$\Delta_{x}$. It follows that the
circular arc~$\pi([p,m])$ with endpoints~$\pi(p)$ and~$\pi(m)$ is
contained in~$H_{m}^{c}$.
\end{proof}

\begin{lemma}[Length estimate for horizontal segments] \label{lem:h}
Let $\sigma:=[p,q]$ be an horizontal segment of~$A$ with midpoint
$m$, pointing in the positive direction. Assume $\dist(O,q)\le
\dist(O,p)$. Then, $$ \length(\sigma) \leq \frac{2}{\lambda} \,
\length(\pi([p,m])). $$
\end{lemma}

\begin{proof}
The line passing through~$p$ and the origin $O$ together with the
circle of radius~$|Op|$ centered at the origin define a
decomposition of the circle of radius~$|pm|$ centered at~$p$ into
four arcs. One of these arcs, denoted by~$C$, contains the
point~$m$. Let $m'$ be the endpoint of~$C$ lying in the circle of
radius~$|Op|$ centered at the origin, \cf~Fig.~\ref{fig:3} below. By
construction,
\begin{equation} \label{eq:1}
\length(\sigma) = 2 \, |pm| = 2\, |pm'|.
\end{equation}
Since $m'$ is at the same distance from the origin as~$p$, there
exists a rotation~$\rho$ centered at the origin which takes~$p$
to~$m'$. This rotation sends the ray~$[O,p]$ to~$[O,m']$ and
preserves distances and angles. Therefore, it also
sends~$\Delta_{p}$ to~$\Delta_{m'}$. In particular, the
rotation~$\rho$ maps~$\pi(p)$ to~$\pi(m')$. From Thales' formula, we
derive
$$
\frac{|\pi(p)\pi(m')|}{|pm'|} = \frac{|O\pi(p)|}{|Op|}.
$$
Hence,
\begin{equation} \label{eq:2}
|\pi(p)\pi(m')| \geq \lambda \, |pm'|.
\end{equation}
Since the endpoints of the segment $[\pi(p),\pi(m')]$ lie in the
arc~$\pi([pm'])$, we have
\begin{equation} \label{eq:3}
|\pi(p)\pi(m')| \leq \length(\pi([p,m'])).
\end{equation}

\setlength\unitlength{1pt} \noindent \begin{figure}[htbp] \noindent
\begin{picture}(370,205)(45,0)
\put(50,0){$\includegraphics[height=200pt]{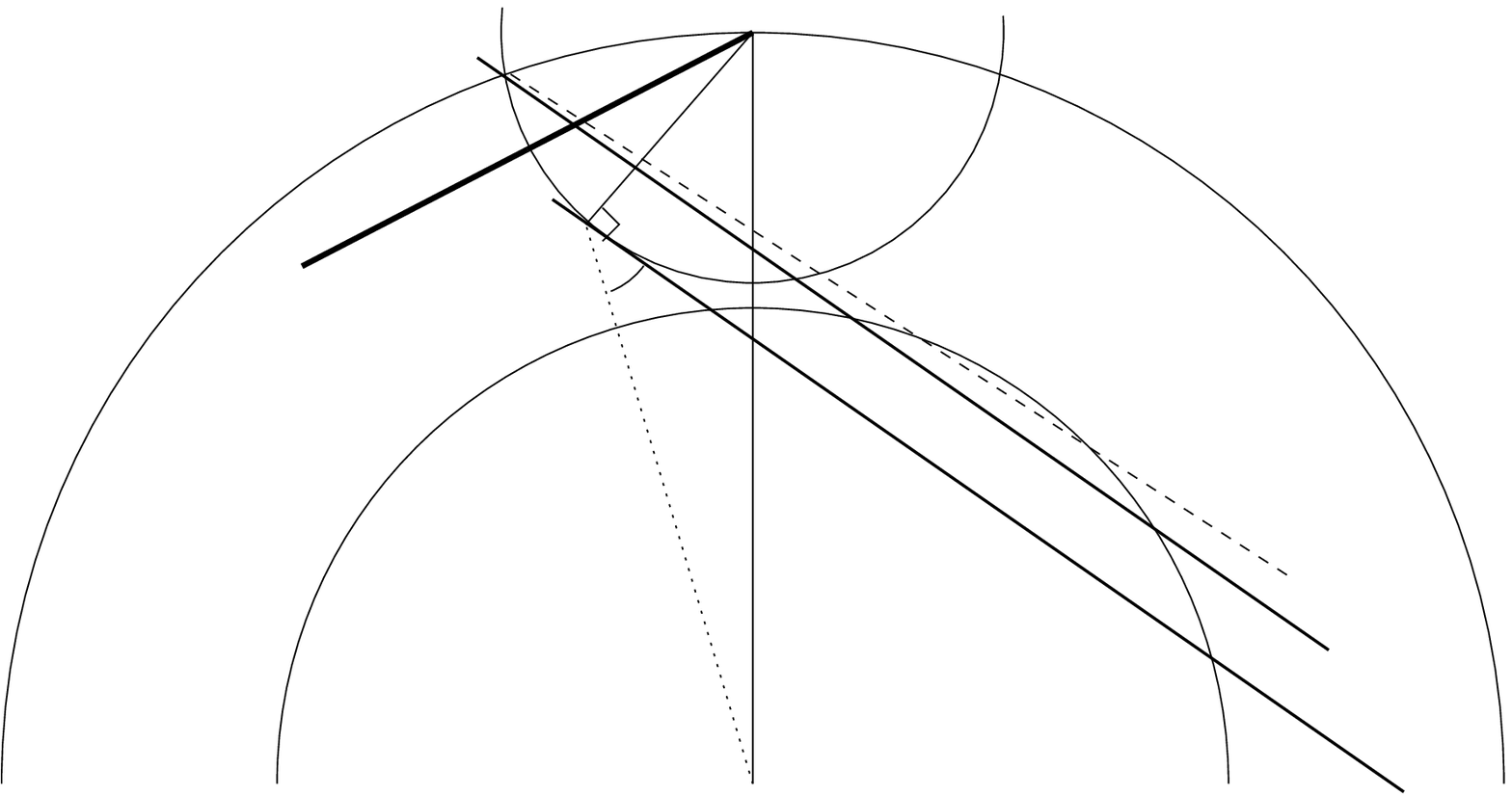}$}
\put(244,-1){$O$} \put(240,200){$p$} \put(117,134){$q$}
\put(211,108){\small $\pi(m'')$} \put(184,132){$m''$}
\put(170,163){$m$} \put(160,190){$m'$} \put(396,10){$D''$}
\put(389,35){$D'$} \put(375,60){$L_{m'}$} \put(129,30){$S(0,\lambda
R)$} \put(209,125){\small $\alpha$}
\end{picture}
\caption{} \label{fig:3}
\end{figure}

When a point~$x$, starting at~$m'$, moves along~$C$, the
angle~$\widehat{Oxp}$ increases from less than~$\frac{\pi}{2}$
to~$\pi$. Thus, there exists a unique point~$m''$ of~$C$ where the
angle~$\widehat{Om''p}$ is equal to~$\frac{\pi}{2}+\alpha$. By
definition of an horizontal segment pointing in the positive
direction, the angle~$\widehat{Omp}$ lies between~$\frac{\pi}{2}$
and~$\frac{\pi}{2}+\alpha$. Therefore, the point~$m$ lies in~$C$
between~$m'$ and~$m''$, \cf~Fig.~\ref{fig:3}. \smallskip

The angles~$\widehat{Om''p}$ and~$\widehat{Om''\pi(m'')}$ are equal
to~$\frac{\pi}{2}+\alpha$ and~$\alpha$. Therefore, the ray~$pm''$
makes a right angle at~$m''$ with the line~$D''$ passing
through~$m''$ and~$\pi(m'')$. Thus, the line~$D''$ is tangent to~$C$
at~$m''$. This implies that the points $O$, $m$ and~$m''$ lie in the
same half-plane delimited by the line~$D'$ passing through~$m'$ and
parallel to~$D''$. Since the ray~$Om'$ makes an angle less
than~$\alpha$ with~$D'$ at~$m'$, the angle between~$D'$ and~$L_{m'}$
(the half-line passing through~$m'$ and making an angle~$\alpha>0$
with~$\overrightarrow{m'O}$) is positive, \cf~Fig.~\ref{fig:3}.
Therefore, the points $O$, $m$ and~$m''$ lie in~$H_{m'}$. \smallskip

By applying Lemma~\ref{lem:nointersec} (Essential disjointness of
$\Delta_x$, $\Delta_y$), successively for $x=m'$ and~$y=m$, and for
$x=p$ and~$y=m'$, we obtain that $\pi([p,m'])$ is contained
in~$\pi([p,m])$ from the injectivity of the restriction of~$\pi$ to
the segments $[p,m']$ and~$[p,m]$, \cf~Lemma~\ref{lem:hc}. Hence,
\begin{equation} \label{eq:4}
\length(\pi([p,m'])) \leq \length(\pi([p,m])).
\end{equation}
Putting together the inequalities \eqref{eq:1}, \eqref{eq:2},
\eqref{eq:3} and \eqref{eq:4}, we obtain the desired bound.
\end{proof}

Let us now consider a polygonal decomposition in~$A$
$$
\sigma_i:=[p_i,q_i], 
\quad i\in\{1,\dots,k\}
$$
of a bounded self-contracted curve~$\gamma$ converging to $O$, \cf~Definition~\ref{def:approx}.

\begin{lemma} [Disjointness of $\pi(p_i,m_i)$ and $\pi(p_j,m_j)$] \label{lem:disj}
Let $\sigma_{i}$ and~$\sigma_{j}$ be two distinct horizontal
segments of a polygonal approximation of~$\gamma$ in~$A$, \cf~Definition~\ref{def:approx},
both pointing in the positive direction. Then,
the images by~$\pi$ of $[p_{i},m_{i}]$ and $[p_{j},m_{j}]$ are
disjoint.
\end{lemma}

\begin{proof}
Switching the indices~$i$ and~$j$ if necessary, we can assume that
$i < j$.

From Lemma~\ref{lem:hc}, the arc~$\pi([p_{i},m_{i}])$ is contained
in~$\overline{H_{m_{i}}^{c}}$. To prove the desired result, it is
enough to show that $\pi([p_{j},m_{j}])$ lies in the complement
of~$\overline{H_{m_{i}}^{c}}$ ({\it i.e.} the interior of
$H_{m_i}$).\smallskip

From the definition of a self-contracted curve, the
points~$p_{j}$ and~$q_{j}$ are closer to~$q_{i}$ than to~$p_{i}$
(see Remark~\ref{rmq-non-sc}). Thus, $p_{j}$ and~$q_{j}$, and so
their barycenter~$m_{j}$, lie in the half-plane delimited by the
perpendicular bisector of~$\sigma_{i}$. (Notice that this half-plane
also contains the origin $O$, in view of
Corollary~\ref{corollary-convergence}.) The half-line of this
bisector with endpoint~$m_{i}$ which makes an acute angle with the
ray~$m_{i}O$ is noted~$D_{m_{i}}$. Since the horizontal
segment~$\sigma_{i}$ points in the positive direction, its
half-bisector~$D_{m_{i}}$ makes an angle less or equal to~$\alpha$
with~$m_{i}O$. Thus, $L_{m_{i}}$ lies in the half-plane delimited by
the perpendicular bisector of~$\sigma_{i}$ not containing the
origin.

\setlength\unitlength{1pt} \noindent \begin{figure}[htbp] \noindent
\begin{picture}(370,150)(45,0)
\put(140,0){$\includegraphics[height=150pt]{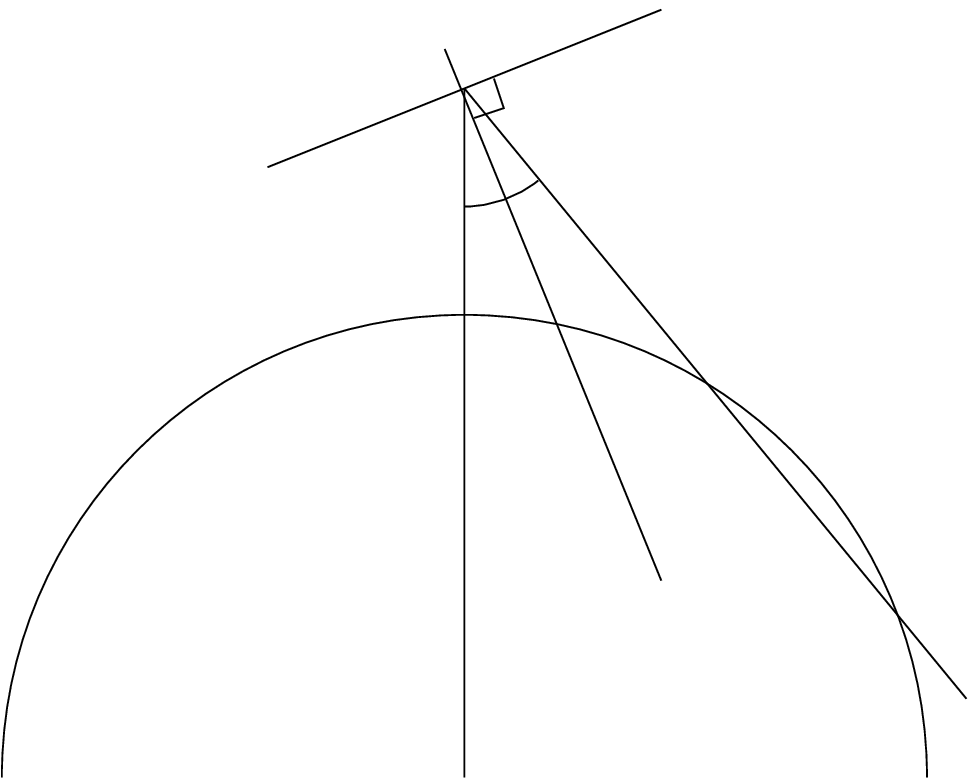}$}
\put(217,-1){$O$} \put(229,144){$m_i$} \put(268,155){$p_i$}
\put(182,122){$q_i$} \put(261,25){$D_{m_i}$} \put(331,11){$L_{m_i}$}
\put(150,20){$S(0,\lambda R)$}
\end{picture}
\caption{} \label{fig:4}
\end{figure}

Now, since the function $t \mapsto \dist(O,\gamma(t))$ is
nonincreasing, the points $p_{j}$ and~$q_{j}$, and so their
barycenter~$m_{j}$, belong to the disk of radius~$|Oq_{i}|
<|Om_{i}|$ centered at the origin. Therefore, the points $p_{j}$
and~$m_{j}$ lie in~$H_{m_{i}}$, and
$\dist(O,m_{i})>\mathrm{max}\{\dist(O,p_j), \dist(O,m_j)\}$. From
Lemma~\ref{lem:nointersec} (Essential disjointness of $\Delta_x$,
$\Delta_y$), the segments $\Delta_{p_{j}}$ and~$\Delta_{m_{j}}$ lie
in ~$H_{m_{i}}$ and do not intersect its boundary. Therefore, their
endpoints $\pi(p_{j})$ and~$\pi(m_{j})$ also lie in the
interior~$H_{m_{i}}$.
\end{proof}

The total length of the horizontal segments satisfies the following
upper bound.

\begin{lemma}[Total horizontal length upper bound] \label{lem:H}
If $\{\sigma_{i}\}_{i=1}^{k+1}$ is a polygonal approximation of
$\gamma$ in the annulus $A$, \cf~Definition~\ref{def:approx}, then
$$ \sum_{i \in
\mathcal{H}} \length(\sigma_{i}) \leq \frac{8 \pi}{1-\lambda} \,
\Delta R
$$
where $\mathcal{H}$ is the set of indices corresponding to the
horizontal segments.
\end{lemma}

\begin{proof}
From Lemma~\ref{lem:h} (Length estimate for horizontal segments),
the sum of the lengths of the horizontal segments~$\sigma_{i}$
pointing in the positive direction satisfies $$ \sum_{i \in
\mathcal{H}_{+}} \length(\sigma_{i}) \leq \frac{2}{\lambda} \,
\sum_{i \in \mathcal{H}_{+}} \length(\pi([p_{i},m_{i}])),$$ where
$\mathcal{H}_{+}$ is the set of indices corresponding to the
horizontal segments pointing in the positive direction. Since the
arcs~$\pi([p_{i},m_{i}])$ of the circle~$S(0,\lambda R)$ are
disjoint, \cf~Lemma~\ref{lem:disj}, we have
$$ \sum_{i \in \mathcal{H}_{+}} \length(\pi([p_{i},m_{i}])) \leq 2
\pi \lambda R. $$ Analogous arguments lead to a similar bound for
the sum of the lengths of the horizontal segments~$\sigma_{i}$
pointing in the negative direction. Recalling that $\Delta R=
(1-\lambda)R$ the result follows.
\end{proof}

\section{Proof of the main result} \label{sec:proofs}

In order to prove our main theorem (\cf~Theorem~\ref{Theorem_main}),
we shall first need the following result.

\begin{proposition} [Length estimate in the annulus $A$] \label{prop:DR}
Every continuous self-contracted planar curve~$\gamma$ converging to
the origin $O$ satisfies $$\length(\gamma \cap A) \leq (8 \pi + 2)
\, \Delta R. $$
\end{proposition}

\begin{proof}
Consider a decomposition of~$\gamma$ into segments~$\sigma_i$ as in
Definition~\ref{def:approx} (refining a subdivision does not decrease the
sum). From Lemma~\ref{lem:V} (Total vertical length upper bound) and
Lemma~\ref{lem:H} (Total horizontal length upper bound), the
length~$L$ of the polygonal line $\sigma_{1} \sigma_{2} \cdots
\sigma_{k}$ satisfies
\begin{eqnarray*}
L & = & \sum_{i \in \mathcal{H}} \length(\sigma_{i}) + \sum_{i \in \mathcal{V}} \length(\sigma_{i}) \\
 & \leq & \left( \frac{8\pi}{1-\lambda} + \frac{2}{\cos \alpha} \right) \, \Delta R.
\end{eqnarray*}
By taking the supremum of~$L$ over all such decompositions
respecting the annulus $A$, we derive the same upper bound for the
length of~$\gamma \cap A$. Finally, by letting $\alpha$ and
$\lambda$ go to zero, we obtain $$ \length(\gamma \cap A) \leq (8
\pi + 2) \, \Delta R. $$
\end{proof}

Now, we can derive our main result.

\begin{proof}[Proof of Theorem~\ref{Theorem_main}]
From Corollary~\ref{corollary-convergence}, the bounded continuous
self-contracted curve~$\gamma$ converges to a point. Using a
translation if necessary, we can assume that this point agrees with
the origin~$O$ of~$\R^2$. \smallskip

Let $t_-=\inf I$. Denote by $\gamma(t_-)$ the limit of~$\gamma(t)$
when $t$ goes to~$t_-$, (\cf~Proposition~\ref{prop:cont} (Existence
of left/right limits)). Set $R_{0}= \dist(O,\gamma(t_-))$. For
$i\in\mathbb{N}$, let~$A_{i}$ be the planar annulus centered at the
origin with outer radius~$R_{i}$ and inner radius~$R_{i+1}$, where
$R_{i+1}=\lambda R_{i}$, with $\lambda\in(0,1)$ given in
\eqref{lambda}. From Proposition~\ref{prop:DR} (Length estimate in
the annulus), we have
\begin{equation} \label{eq:Di}
\length(\gamma \cap A_{i}) \leq (8\pi+2) \, \Delta R_{i}
\end{equation}
where $\Delta R_{i}$ is the width of~$A_{i}$. Since $\lambda < 1$,
the sequence~$R_{i}$ goes to zero and the sum of the width of the
disjoint annulus~$A_{i}$ is equal to~$R_0$. Thus, taking the sum of
the above inequalities~\eqref{eq:Di} for $i\in\mathbb{N}$ we obtain
the desired bound $$ \length(\gamma) \leq (8\pi+2) \,
\dist(O,\gamma(t_-)).$$ The proof is complete.
\end{proof}

\section{Gradient and subgradient systems, and convex foliations}

\label{sec:6}

In this section, we apply Theorem \ref{Theorem_main} to derive
length estimates, first for orbits of dynamical systems of gradient
or subgradient type, then for orbits orthogonal to a convex
foliation. The key fact is to observe that in some interesting
particular cases (for instance, $f$ convex or quasiconvex) these
curves are self-contracted. Recall however that this is not the case
for gradient dynamical systems defined by a general $C^{\infty}$
function, as already observed in Remark~\ref{remark-basic}~(iv).

\subsection{Gradient dynamical system -- quasiconvex case}

Let $f:\R^n\rightarrow\R$ be a $C^{k}$ function ($k\geq1$),
$x_{0}\in\R^n$ and consider the Gradient Dynamical System
\begin{equation} \left\{
\begin{array}
[c]{l} \dot{\gamma}(t)=-\nabla f(\gamma(t)), \ t>0 \smallskip \\
\gamma(0)=x_{0}\in\R^n.
\end{array}
\right.  \label{GDS}
\end{equation}
It follows directly from the standard theory of Ordinary
Differential Equations (see \cite{Hale1980}, for example) that the
system~\eqref{GDS} admits a solution (trajectory)
$\gamma:I\longmapsto \mathbb{R}^{n}$, where
$I\subset\lbrack0,+\infty)$, which is a curve of class $C^{k-1}.$
Note that the case $k=1$ corresponds to mere continuity of $\gamma
$, while for $k>1$ (or more generally, if $f$ is assumed $C^{1,1},$
that is, $\nabla f$ is Lipschitz continuous), the trajectory
$\gamma$ is unique. In the sequel, we shall always consider maximal
solutions of \eqref{GDS}, that is, for which $I=[0,T),$ where $T>0$
is the maximal time such that $\gamma$ is defined in $[0,T).$ We
shall refer to them as orbits of the \emph{gradient flow} of $f$.
\smallskip

We will also need the following definition.

\begin{definition}[Convex, quasiconvex and coercive functions] 
A function $f:\mathbb{R}^{n}\rightarrow\mathbb{R}$ is called
\emph{convex} (respectively, \emph{quasiconvex}) if its epigraph
\[
\mathrm{epi\,}f:=\{(x,y)\in\mathbb{R}^{n+1} \mid f(x)\leq y\}
\]
is a convex subset of $\mathbb{R}^{n} \times \R$ (respectively, if
for every $y\in\mathbb{R}$ the sublevel set $\{x\in\mathbb{R}^{n}
\mid f(x)\leq y\}$ is a convex subset of $\mathbb{R}^{n}$). A
function $f$ is called \emph{coercive} (or \emph{proper}), if its
level sets are bounded, or equivalently, if
\begin{equation}\label{coercive}
\lim_{\|x\|\rightarrow +\infty} f(x)=+\infty.
\end{equation}
\end{definition}

It is straightforward to see that whenever $f$ is coercive 
the corresponding flow orbits are bounded curves and therefore
$I=[0,+\infty)$ (the trajectories are defined for all $t\ge 0$).
Notice in particular that the function
\begin{equation*}
t\longmapsto f(\gamma(t)),\quad t\in [0,+\infty)\label{L}
\end{equation*}
is a natural Lyapunov function for the orbits of the flow, \ie~it is
nonincreasing along the trajectories. Moreover, unless $\gamma$
meets a critical point (\textit{i.e.}~$\nabla f(\gamma(t_{\ast}))=0$
for some $t_{\ast}\in [0,+\infty)$), the function defined
in~\eqref{L} is decreasing and $\gamma$ is injective. \smallskip

Let us finally recall (\textit{e.g.}~\cite[Theorem~2.1]{yboon}) that
a (differentiable) function $f:\mathbb{R}^n\rightarrow\mathbb{R}$ is
quasiconvex if and only if for every $x,y \in \mathbb{R}^n$ the following
holds:
\begin{equation}\label{qc}
\langle \nabla f(x), y-x\rangle >0 \Rightarrow f(y)\geq f(x).
\end{equation}
We are now ready to establish the following result.

\begin{proposition}
[Quasiconvex orbits are self-contracted
curves]\label{proposition_GDS}The orbits of the gradient flow of a
quasiconvex $C^{1,1}$ function are self-contracted curves.
\end{proposition}

\begin{proof} Let $\gamma:I\longmapsto\mathbb{R}^{n}$ be an orbit of
the gradient flow of $f$. Let $0\leq t\leq t_1$ be in $I$ and
consider the function
\[
g(t)=\frac{1}{2}||\gamma(t)-\gamma(t_1)||^{2},\quad t\in I.
\]
In view of \eqref{GDS}, we easily deduce that
\[
g^{\prime}(t)=\langle\nabla
f(\gamma(t)),\gamma(t_1)-\gamma(t)\rangle.
\]
If $g^{\prime}(t)>0$ for some $t\in\lbrack0,t_{1}),$ then the
quasiconvexity of $f$ would imply that $f(\gamma(t_{1}))\geq
f(\gamma(t))$ (see \eqref{qc}), which in view of \eqref{L} would
yield that $\gamma(t^{\prime })=\gamma(t_{1})$ for all
$t^{\prime}\in\lbrack t,t_{1}]$ and $\nabla f(\gamma(t))=0,$ a
contradiction. Thus, $g$ is nonincreasing in the interval
$[0,t_{1}]$, which proves the assertion. \end{proof}

The following corollary is a straightforward consequence of the
previous proposition and Theorem~\ref{Theorem_main} (Main result).

\begin{corollary}
[Orbits of a gradient quasiconvex flow]\label{Corollary_quasiconvex}
Let $f:\R^2\rightarrow\R$ be a coercive $C^{1,1}$ quasiconvex
function.
Then, for every $x_0 \in \R^2$, the orbit of the gradient flow~\eqref{GDS} converges and has finite length.

\forget
 Let $\gamma:[0,+\infty)\mapsto\mathbb{R}^{2}$ be an orbit
of the corresponding gradient flow with $\gamma(0)=x_{0}\in \R^2$.
Then \smallskip
\begin{enumerate}
\item[(i)]  $\displaystyle \lim_{t \to +\infty} \gamma(t)=x_{\infty}\in \R^2$ (convergence of the
orbit); \smallskip
\item[(ii)] $\displaystyle
  \int_{0}^{+\infty}||\dot{\gamma}(t)||dt\leq(8\pi+2)
||x_{0}-x_{\infty}||$ (finite length of the orbit).
\end{enumerate}
\forgotten

\end{corollary}


\subsection{Subgradient dynamical systems - convex case}

A convex function is a particular case of a quasiconvex function.
Therefore, Corollary~\ref{Corollary_quasiconvex} implies that the
orbits of the gradient flow of $C^{1,1}$ convex functions are of
finite length. It is well-known (\cite{Brezis}) that in the case of
a (nonsmooth) convex function $f:\R^n\to\R$ (or more generally, for
a semiconvex function \cite{DMT85}), the gradient system~\eqref{GDS}
can be generalized to the following differential inclusion, called
Subgradient Dynamical System
\begin{equation}
\left\{
\begin{array}
[c]{l} \dot{\gamma}(t)\in-\partial f(\gamma (t))\text{
  \quad}\mathrm{a.e.} \ t\in [0,+\infty),\\
\gamma (0)=x_{0}\in\R^n,
\end{array}
\right.  \label{SDS}
\end{equation}
where $\partial f$ is the set of the subgradients (subdifferential)
of $f.$ If $f:\R^n\to\R$ is convex, then this latter set is defined
as
\[
\partial f(x)=\{p\in\mathbb{R}^{n} \mid  f(y)\geq f(x)+\langle
p,y-x\rangle,\;\forall y\in\mathbb{R}^{n}\} \quad {\rm for \ all \ }
x\in\R^n\,.
\]
The above formula defines always a nonempty convex compact subset of
$\R^n$, and reduces to $\{\nabla f(x)\}$ whenever $f$ is
differentiable at $x$, {\it cf.} \cite{Clarke83}. It is also known
that \eqref{SDS} has a unique absolutely continuous solution
$\gamma: [0,+\infty)\to \R^n$, that is, the derivative
$\dot{\gamma}(t)=\frac{d}{dt}\gamma(t)$ exists almost everywhere and
for every $0\leq t_{1}\leq t_{2},$
\[
\gamma(t_{2})=\gamma(t_{1})+\int_{t_{1}}^{t_{2}}\dot{\gamma}(t)\,dt
\,.
\]

\forget
When $n=2$, we have the following generalization of Theorem~\ref{thm-intro-convex}.

\begin{proposition}
[Orbits of a subgradient convex flow] 
Let $f:\mathbb{R}^{2}\rightarrow\mathbb{R}$ be a convex continuous
function which admits a minimum. Then, for every
$x_0\in\R^2,$ the trajectory~$\gamma$ of the subgradient system \eqref{SDS}
is a bounded self-contracted curve and therefore has a finite
length.
\end{proposition}
\forgotten

The analogous of Proposition~\ref{proposition_GDS} holds true.

\begin{proposition}
[Subgradient convex flow orbits are self-contracted curves] Let
$f:\mathbb{R}^{n}\rightarrow\mathbb{R}$ be a convex continuous
function.
Then, for every $x_0\in\R^n,$ the trajectory~$\gamma$ of the subgradient system~\eqref{SDS}
is a self-contracted curve.
\end{proposition}


\begin{proof}
We give a sketch of proof for the reader convenience (we refer to
\cite{Brezis} for details). It is easy to prove that $t\in
[0,+\infty)\mapsto f(\gamma (t))$ is convex and that for almost all
$t\geq 0$ we have
$$
\frac{d}{dt}f(\gamma (t))= -||\dot\gamma(t)||^2\leq 0.
$$
Therefore $t\mapsto f(\gamma (t))$ is nonincreasing and
$\gamma(t)\in \{f\leq f(x_0)\}$ is bounded. Moreover, for all
$t_1>0$ and for almost all $t\in [0,\,t_1]$
$$
\frac{1}{2}\, \frac{d}{dt} ||\gamma(t)-\gamma(t_1)||^2= \langle
\dot\gamma(t), \gamma(t)-\gamma(t_1)\rangle \leq
f(\gamma(t_1))-f(\gamma(t))\leq 0.
$$
This implies that $t\in [0,t_1)\mapsto  ||\gamma(t)-\gamma(t_1)||^2$
is nonincreasing yielding that $\gamma$ is a self-contracted curve.
\end{proof}

When $n=2$, we have the following generalization of
Theorem~\ref{thm-intro-convex}.

\begin{corollary}[Orbits of a subgradient convex flow] \label{Corollary_convex}
Let $f:\R^{2} \rightarrow \mathbb{R}$ be a convex continuous function which admits a minimum.
Then, for every $x_0 \in \R^2$, the orbit of the gradient flow~\eqref{SDS} converges and has finite length.
\end{corollary}


\subsection{Trajectories orthogonal to a convex foliation}
\label{sec:7}

In this section we consider orbits that are ``orthogonal" to a
\emph{convex foliation}. Let us introduce the definition of the
latter. (For any subset $C\subset \R^2,$ ${\rm int}\, C$ denotes the
interior of $C$ and $\partial C$ its boundary.)  \smallskip

Let $\{C_{\alpha}\}_{\alpha\in [0,A]}$ (where $A>0$) be a family of
subsets of $\R^2$ such that
\begin{eqnarray} \label{def-foliation}
\begin{array}{l}
\text{(i) For all $\alpha\in [0,A],$ $C_\alpha$ is convex compact}
\smallskip \\
\text{(ii) If $\alpha >\alpha',$ then $C_\alpha\subset {\rm int}\, C_{\alpha '},$}\smallskip \\
\text{(iii) For every $x\in C_0\setminus {\rm int}\,C_A,$ there
exists a unique $\alpha\in [0,A]$ such that $x\in \partial
C_\alpha.$}
\end{array}
\end{eqnarray}
We shall refer to the above as a foliation made up of convex
surfaces. We shall now define a notion of ``orthogonality" for an
orbit $\gamma$ with respect to this foliation. To this end, let
$T\in (0,+\infty]$ and $\gamma : [0,T)\to\R^2$ be an absolutely
continuous curve. We say that the curve $\gamma$ is ``orthogonal" to the
foliation defined in~\eqref{def-foliation} if the following
conditions hold:
\begin{eqnarray}
\label{def-courbe-orth}
\begin{array}{l}
\text{(i) for every $t\in [0,T),$ there exists $\alpha\in [0,A]$
such that $\gamma (t)\in \partial C_\alpha,$}\smallskip \\
\text{(ii) for almost all $t\in (0,T)$, if $\gamma (t)\in \partial
C_\alpha,$ then for all $x\in C_\alpha$, $\langle \dot\gamma (t), x-\gamma(t)\rangle\geq 0,$}\smallskip \\
\text{(iii) if $t'>t$ and $\gamma(t)\in  C_\alpha ,$ then
$\gamma(t')\in C_{\alpha }.$}
\end{array}
\end{eqnarray}

Condition (ii) in \eqref{def-courbe-orth} is a nonsmooth
generalization of orthogonality: if $\partial C_\alpha$ is smooth at
$\gamma(t)$ and $\gamma$ is differentiable at $t$ then
$\dot\gamma(t)$ is orthogonal to the tangent space of $\partial
C_{\alpha}$ at $\gamma(t)$. Further, condition (iii) guarantees that the curve $\gamma (t)$ enters into each convex set of the
foliation. In this context, one has the following result.

\begin{proposition}
[Orbits orthogonal to a convex foliation]
\label{Proposition_foliation} The curve $\gamma$ is a bounded
self-contracted curve, thus, of bounded length.
\end{proposition}

\begin{proof}
The curve $\gamma$ is clearly bounded. Let $0\leq t_1<T$. Then, for
almost all $t\in [0,t_1]$, we have
\begin{eqnarray}
\label{form123} \frac{1}{2}\, \frac{d}{dt}
||\gamma(t)-\gamma(t_1)||^2= \langle \dot\gamma(t),
\gamma(t)-\gamma(t_1)\rangle.
\end{eqnarray}
By \eqref{def-courbe-orth} (i), we have $\gamma(t)\in \partial C_\alpha$ for
some $\alpha.$ By \eqref{def-courbe-orth} (iii) and since $t_1\geq
t,$ we also have $\gamma (t_1)\in C_\alpha.$ Therefore, \eqref{def-courbe-orth}
(ii) implies that the right-hand side of \eqref{form123} is
nonpositive. It follows that $t\in [0,t_1]\mapsto
||\gamma(t)-\gamma(t_1)||^2$ is nonincreasing and $\gamma$ is
self-contracted. Applying Theorem~\ref{Theorem_main} finishes the
proof. \end{proof}

\begin{remark} (i) The sublevel sets of a continuous quasiconvex
function need not define a convex foliation. Indeed, consider the
quasiconvex function $f:[-2,2]\rightarrow\mathbb{R}$ given by
\begin{equation*}
f(x)=\left\{
\begin{array}{cc}
x, & \text{if }  -2\leq x\leq 0, \\
0, & \text{if } \phantom{-}0\leq x\leq 1, \\
x-1, & \text{if } \phantom{-}1\leq x\leq 2.
\end{array} \right.
\end{equation*}
Then the sublevel sets of $f$ do not define a foliation on
$[-2,2]\subset\mathbb{R}$ since property (iii) of
\eqref{def-foliation} fails at the level set $[f=0]$. This drawback
appears whenever the level sets of such functions have ``flat" parts
outside the set of their global minimizers. Actually,
it follows from \cite[Theorem~3.1]{yboon} that the sublevel sets of
a continuous quasiconvex coercive function $f$ define a convex
foliation if and only if the function is \emph{semi-strictly
quasiconvex}. (We refer to \cite{yboon} for the exact definition and
basic properties of semi-strictly quasiconvex functions.) \smallskip

(ii) Let $f:\R^2\to\R$ be a coercive $C^{1,1}$ quasiconvex function
and $\gamma : [0,+\infty)\to\R^2$ be the solution of \eqref{GDS}.
Let $x_\infty$ be the limit of $\gamma(t)$ as $t\to +\infty$ and
assume that $f$ has no critical point in $\{f(x_\infty)< f\leq
f(x_0)\}$. Then, it is not difficult to see that $\{f\leq
\alpha\}_{\alpha\in [f(x_\infty),f(x_0)]}$ is a family of $C^1$
convex compact subsets which satisfies \eqref{def-foliation} (in
fact, $f$ is semi-strictly quasiconvex in $[f(x_\infty),f(x_0)]$) and
$\gamma$ satisfies \eqref{def-courbe-orth}. \smallskip

(iii) Despite the first remark,
Proposition~\ref{Proposition_foliation} can be used to obtain the
result of Corollary~\ref{Corollary_quasiconvex} without the extra
assumption made in the second remark. The reason is that the trajectory of the
gradient flow will not pass through the flat parts of $f$ anyway (if
it does, then it stops there). We leave the technical details to the
reader.
\end{remark}

\section{Two counter-examples}\label{sec:ex}

\subsection{Absence of Convexity} \label{subsec:1}
The second conclusion of Corollary~\ref{Corollary_quasiconvex} fails
if $f$ is not quasiconvex, even when the function is $C^{\infty}$
and has a unique critical point at its global minimum. Let us give
an explicit example below: \smallskip

Define a function~$f:\R^{2} \to \R$ in polar coordinates as
$$
f(r,\theta) = e^{-1/r} (1+r+\sin(\tfrac{1}{r}+\theta))
$$
with $f(O)=0$. The graph of~$f$ in the plane $\theta = 0$ looks like
the graph of Fig.~\ref{fig:5}.
\setlength\unitlength{1pt} \noindent \begin{figure}[htbp] \noindent
\begin{picture}(370,100)(45,0)
\put(160,0){$\includegraphics[height=100pt,width=150pt]{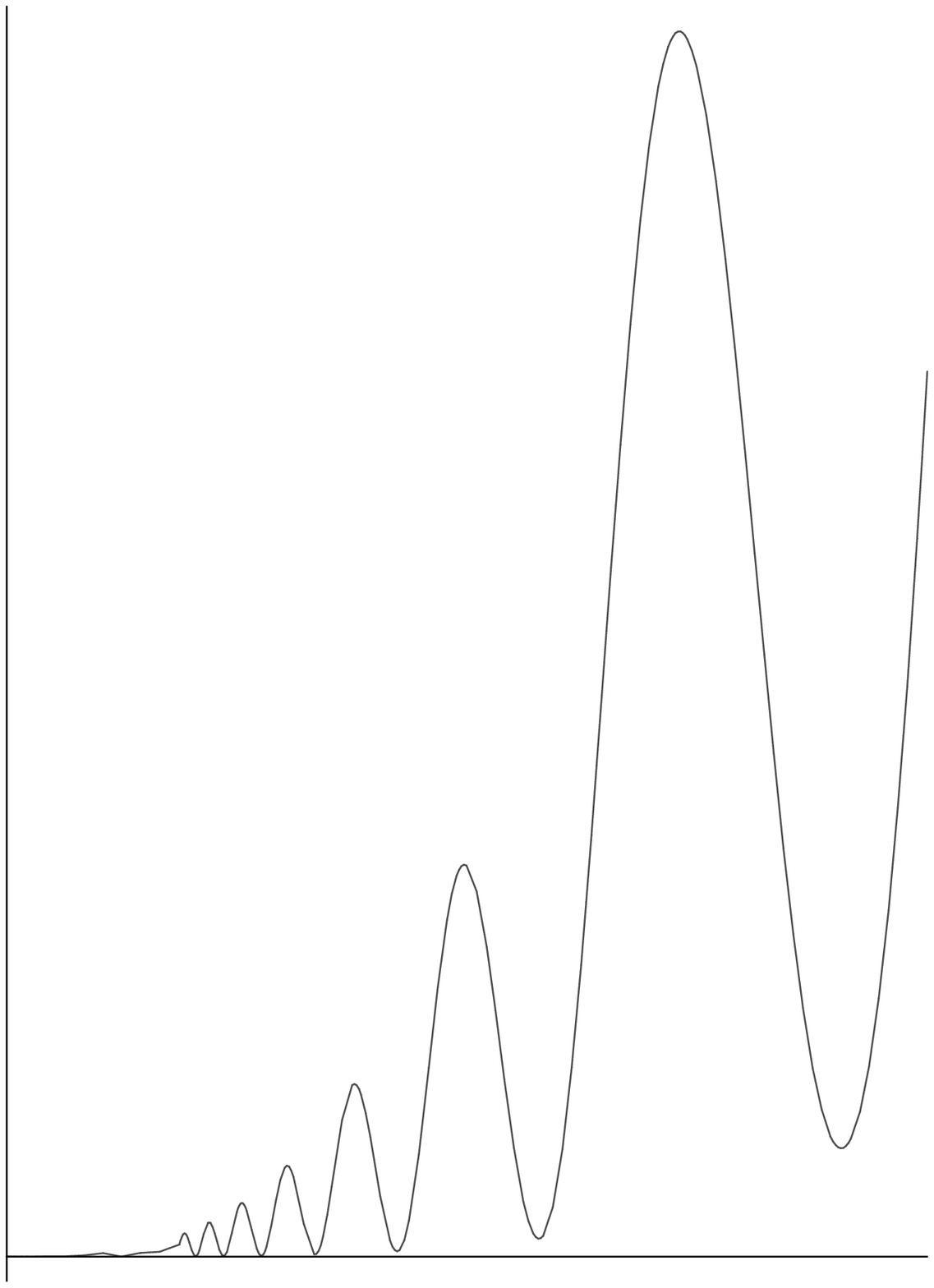}$}
\end{picture}
\caption{} \label{fig:5}
\end{figure}

One can check that $f$ is smooth, positive away from~$O$, with no
critical point except at the origin. The gradient trajectory of~$f$
issued from the point~$(r,\theta)=((\frac{3\pi}{2})^{-1},0)$ remains
close to the spiral given by
\begin{equation*}
\left\{
\begin{array}{rcl}
r & = & \left( \frac{3\pi}{2}+t \right)^{-1} \\
\theta & = & -t
\end{array}
\right.
\end{equation*}
where $t$ runs over~$[0,\infty).$ Therefore, it converges to the
origin and its length is infinite.

\subsection{Thom conjecture fails for convex functions}\label{subsec:2}
Let $f:\mathbb{R}^2\rightarrow\mathbb{R}$ be a convex continuous
function which admits a minimum. Then,
Corollary~\ref{Corollary_convex} guarantees that the orbits of the
gradient flow of $f$ have finite length (thus, \textit{a fortiori},
are converging to the global minimum of $f$). However, it may happen
that each orbit turns around its limit infinitely many times. In
particular the corresponding statement of the Thom conjecture fails
in the convex case. \smallskip

We construct below a counter-example using a technique due to
D.~Torralba~\cite{Torralba} which allows us to build a convex function
with prescribed level-sets given by a sequence of nested convex
sets. Let us recall his result. \smallskip

For any convex set $C\subset\R^n,$ the support function of $C$ is
defined as $\delta_C(x^*)=\sup_{x\in C}\langle x,x^*\rangle$ for all
$x^*\in \R^n.$ Let $\{C_k\}_{k\in\N}$ be a decreasing sequence of
convex compact subsets of $\R^2$ such that $C_{k+1}\subset {\rm
int}\, C_k.$ Set
$$
K_k=\max_{||x^*||=1}\frac{\delta_{C_{k-1}}(x^*)-\delta_{C_k}(x^*)}{
\delta_{C_{k}}(x^*)-\delta_{C_{k+1}}(x^*)}.
$$
Then Torralba's theorem~\cite{Torralba} asserts that for every real sequence
$\{\lambda_k\}_{k\in\N}$ satisfying
\begin{eqnarray}\label{cond-lambda}
0<K_k(\lambda_k-\lambda_{k+1})\le \lambda_{k-1}-\lambda_k \quad \mbox{for every } k\geq 1,
\end{eqnarray}
there exists a continuous convex function $f$ such that for every
$k\in\N,$ $\{f\le \lambda_k\}=C_k$. Moreover, $\lambda_k$ converges
to ${\rm min}\, f$ and, for any $k\geq 0$ and
$\lambda\in\lbrack\lambda_{k+1},\lambda_{k}],$ we have
\begin{equation}\label{interpolation-convexe}
\{ f\leq\lambda\}=\left( \frac{\lambda-\lambda_{k+1}}{\lambda_{k}
-\lambda_{k+1}}\right) C_{k}+\left(
\frac{\lambda_{k}-\lambda}{\lambda_{k}-\lambda_{k+1}}\right) C_{k+1}
\end{equation}
({\it i.e.,} the level-sets of $f$ are convex interpolations of the
two
nearest prescribed level-sets).\\
\setlength\unitlength{1pt} \noindent \begin{figure}[htbp] \noindent
\begin{picture}(440,250)(45,0)
\put(140,0){$\includegraphics[height=250pt]{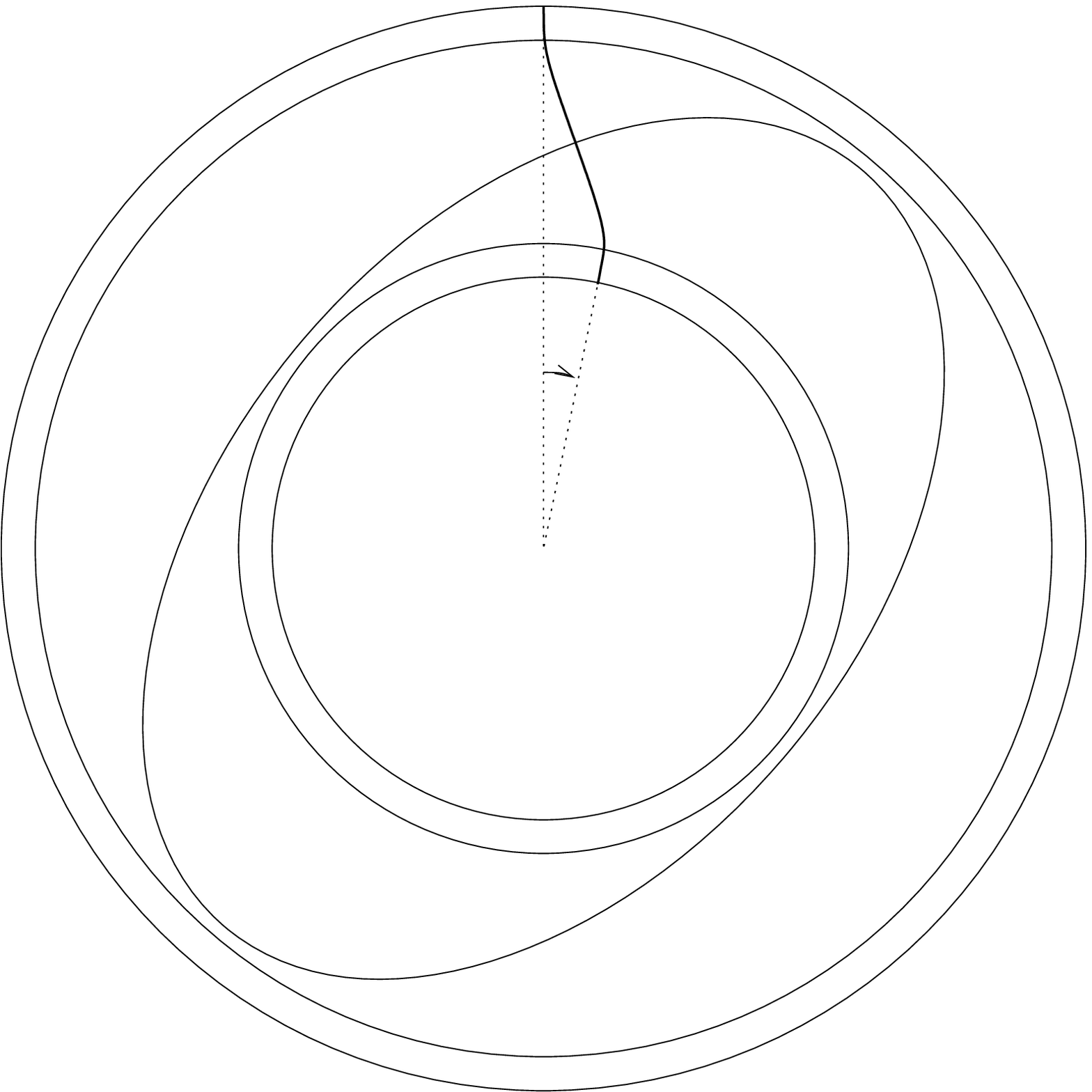}$}
\put(266,168){$\theta$} \put(258,253){$A_0$} \put(276,173){$A_4$}
\put(278,205){$\gamma_0$} \put(261,115){$O$} \put(130,85){$C_0$}
\put(153,109){$C_1$} \put(175,80){$C_2$} \put(205,64){$C_3$}
\put(225,80){$C_4$}
\end{picture}
\caption{} \label{fig:6}
\end{figure}

\noindent{\it Step 1. Constructing a first piece of trajectory.}
Consider the finite decreasing sequence of convex sets $C_0=B(O,1),$
$C_1=B(O,0.9),$ $C_2=E,$ $C_3=B(O,0.6)$ and $C_4=B(O,1/2)$ where $E$
is a compact set bounded by an ellipse (see Fig.~\ref{fig:6}). It is easy to find a sequence
$\{\lambda_k\}$ which satisfies \eqref{cond-lambda}: set $K={\rm
max}\{K_1, K_2,K_3,K_4\}+1> 1$ (since $C'\subset {\rm int}\, C$
implies $\delta_{C}>\delta_{C'}$), take $\lambda_0=1,$
$\lambda_1=1/2$ and
\begin{eqnarray}\label{def-lambda}
\lambda_k-\lambda_{k+1}=\frac{1}{K^k}(\lambda_0-\lambda_1).
\end{eqnarray}
We then obtain a positive function $f_0:C_0\to\R$ with ${\rm
argmin}\,f_0:= \{f_0= {\rm min}\,f_0\} = C_4.$ Consider the subgradient trajectory $\gamma_0$
starting from the point $A_0$ of $C_0$ (see Fig.~\ref{fig:6}). It
reaches $A_4\in\partial C_4.$ From \eqref{interpolation-convexe}
this trajectory is radial (pointing towards the origin) between
$\partial C_0=\{f_0=\lambda_0\}$ and $\partial
C_1=\{f_0=\lambda_1\}$ and between $\partial C_3=\{f_0=\lambda_3\}$
and $\partial C_4=\{f_0=\lambda_4\}.$ Due to the presence of the
ellipse $C_2,$ the trajectory deflects with an angle $\theta:=
\widehat{A_0 O A_4}>0$
in the clockwise direction. \\

\noindent{\it Step 2. Construction of the function from the previous
step.} Consider the transformation $\mathcal{T}=r\circ h,$ where $h$
is the homothety of center $O$ and coefficient $1/2$ and $r$ is the
rotation of center $O$ and angle $\theta$. We define, for
all $k\in\N$ and $\bar k\in \{0,1,2,3\}$
\begin{eqnarray*}
C_k= \mathcal{T}^{[k/4]}(C_{\overline{k}}) \quad \text{where $[k/4]$
is the integer part of $k/4$ and $\overline{k}=k$ (modulo $4$)}
\end{eqnarray*}
(see Fig.~\ref{fig:7} for the first steps of the construction).
\setlength\unitlength{1pt} \noindent \begin{figure}[htbp] \noindent
\begin{picture}(440,250)(45,0)
\put(140,0){$\includegraphics[height=250pt]{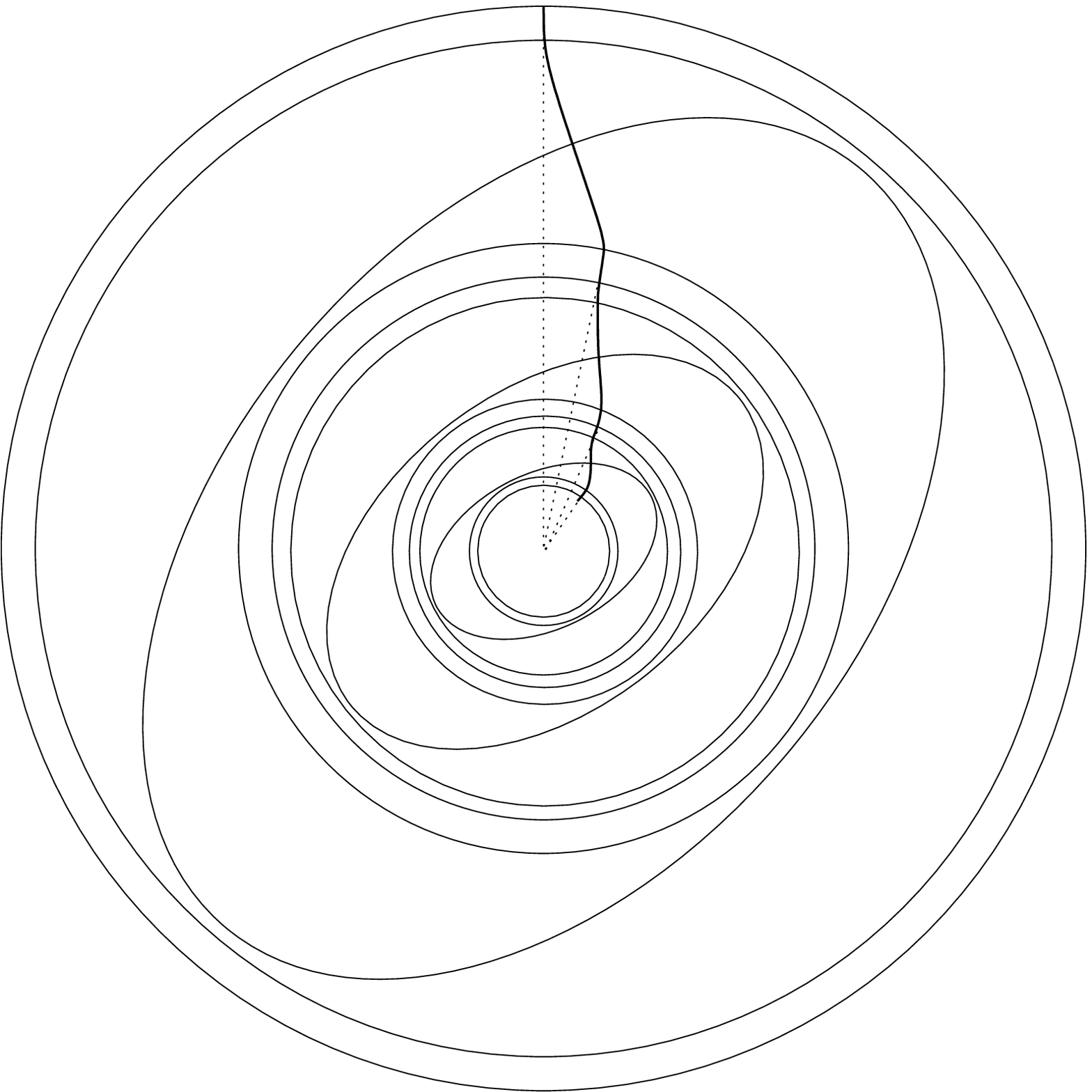}$}
\end{picture}
\caption{} \label{fig:7}
\end{figure}

\noindent The sequence of convex sets $\{C_k\}$ satisfies the
assumptions of Torralba's theorem and we can define a sequence
$\{\lambda_k\}$ as in \eqref{def-lambda} which satisfies
\eqref{cond-lambda} (note that $\{K_k\}$ is 4--periodic since, for
all convex set $C\subset \R^2$ and $x^*\in\R^2,$
$\delta_{\mathcal{T}(C)}(x^*)=\frac{1}{2}\delta_{C}(r^{-1}(x^*))$).
We obtain a convex continuous function $f:C_0\to \R^+$ with ${\rm
argmin}\,f=\{O\}.$ The trajectory starting from the top of $C_0$
spirals around the origin and converges to $O$ (see Fig.~\ref{fig:7}
where the beginning
of the trajectory is drawn with a deflection of $3\theta$).\\

\noindent{\it Step 3. Smoothing $f$.} Actually, the function $f$
built above is $C^\infty$ except at the origin and at the boundaries
$\partial C_k$. It is possible to smooth out $f$ in order to obtain a
function which is $C^\infty$ everywhere except at the origin and
$C^m$ at the origin (for any fixed $m\geq 1$). The smoothing procedure
is quite involved from a technical point of view and is omitted. We
refer the interested reader to \cite[Section~4.3]{BDLM-2008} where
such a smoothing is realized (in a different context). This
procedure does not modify significantly neither the function nor its
gradient trajectories ({\it i.e.} they remain a spiral). This
concludes the construction.

\bigskip

\noindent\textbf{Acknowledgment} This work has been realized during
a research stay of the first author in the University of Tours
(Spring 2008). The authors acknowledge useful discussions with
G.~Barles (Tours), J.~Bolte (Paris~6), H.~Giacomini (Tours),
N.~Hadjisavvas (Aegean) and M.~Hassaine (Talca).


\bigskip

\begin{thebibliography}{99}

\bibitem{AMA2005}\textsc{Absil, P.-A., Mahony, R. \& Andrews, B.,} Convergence
of the Iterates of Descent Methods for Analytic Cost Functions,
\textit{SIAM J. Optim.} \textbf{16} (2005), 531--547.

\bibitem{baillon78}\textsc{Baillon, J.-B.,} Un exemple concernant le
comportement asymptotique de la solution du probl\`{e}me
{$du/dt+\partial \varphi(u)\ni0$}, \textit{\ J. Funct. Anal.}
\textbf{28} (1978), 369--376.

\bibitem{BDL04}\textsc{Bolte, J., Daniilidis, A. \& Lewis, A.,} The \L
ojasiewicz inequality for nonsmooth subanalytic functions with
applications to subgradient dynamical systems, \textit{SIAM J.
Optim.} \textbf{17} (2007), 1205--1223.

\bibitem{BDLS-2008}\textsc{Bolte, J., Daniilidis, A., Lewis, A \& Shiota, M.,
}Clarke subgradients of stratifiable functions, \textit{SIAM J.
Optim.} \textbf{18} (2007), 556--572.

\bibitem{BDLM-2008}\textsc{Bolte, J., Daniilidis, A., Ley, O. \& Mazet, L.,}
Characterizations of \L ojasiewicz inequalities and applications,
57p., CRM Preprint No. 792 (March 2008).

\bibitem{Brezis}\textsc{Br\'ezis, H.,} Op\'{e}rateurs maximaux monotones et
semi-groupes de contractions dans les espaces de Hilbert (French),
\textit{\ North-Holland Mathematics Studies} \textbf{5},
North-Holland Publishing Co., 1973.

\bibitem{bruck75}\textsc{Bruck, R.}, Asymptotic convergence of nonlinear
contraction semigroups in {H}ilbert space, \textit{J. Funct. Anal.}
\textbf{18} (1975), 15--26.

\bibitem{Clarke83}\textsc{Clarke, F.,} \textit{Optimization and non-smooth
analysis}, Wiley Interscience, New York, 1983 (Re-published in 1990:
Vol. 5, Classics in Applied Mathematics, Society for Industrial and
Applied Mathematics, Philadelphia, Pa.).

\bibitem{yboon}\textsc{Daniilidis, A., Garcia Ramos, Y.,} Some remarks on the class of
continuous (semi-)strictly quasiconvex functions, \emph{J. Optim.
Theory Appl.} \textbf{133} (2007), 37--48.

\bibitem{DMT85}\textsc{Degiovanni, M., Marino, A., Tosques, M.,} Evolution
equations with lack of convexity, \emph{Nonlinear Analysis}
\textbf{9} (1985), 1401-1443.

\bibitem{Dries-Miller96}\textsc{Van den Dries, L. \& Miller, C.,} Geometries
categories and o-minimal structures, \textit{Duke Math. J.}
\textbf{84} (1996), 497-540.

\bibitem{Hale1980}\textsc{Hale, J.,} \textit{Ordinary Differential Equations},
Robert E. Krieger Publishing Co., 1980.


\bibitem{Kur98}\textsc{Kurdyka, K.,} On gradients of functions definable in
o-minimal structures, \textit{Ann. Inst. Fourier }\textbf{48}
(1998), 769-783.

\bibitem{KMP2000}\textsc{Kurdyka, K., Mostowski, T. \& Parusinski, A.,} Proof
of the gradient conjecture of R. Thom, \textit{Annals of
Mathematics} \textbf{152} (2000), 763-792.

\bibitem{Loja93}\textsc{\L ojasiewicz, S.,} Sur la g\'{e}om\'{e}trie semi- et
sous-analytique, \textit{Ann. Inst. Fourier }\textbf{43} (1993),
1575-1595.

\bibitem{PalDem82}\textsc{Palis, J. \& De Melo, W.,} \textit{Geometric theory
of dynamical systems. An introduction,} (Translated from the
Portuguese by A. K. Manning), Springer-Verlag, New York-Berlin,
1982.

\bibitem{Thom89}\textsc{Thom, R.,} Probl\`{e}mes rencontr\'{e}s dans mon
parcours math\'{e}matique: un bilan, \textit{IHES Publ. Math.}
\textbf{70} (1989), 199--214.

\bibitem{Torralba}\textsc{Torralba, D.}, \textit{Convergence \'{e}pigraphique
et changements d'\'{e}chelle en analyse variationnelle et
optimisation}, Th\`{e}se, Universit\'{e} de Montpellier 2, 1996.

\end{thebibliography}
\end{document}